\documentclass[dvipdfmx]{amsart}

\usepackage{amsmath, amsfonts,amssymb}
\usepackage{graphicx}
\usepackage{url}
\usepackage[myheadings]{fullpage}
\usepackage[dvipdfm, usenames]{color}
\usepackage{mathrsfs}

\usepackage{comment}

\newcommand{\bN}{\mathbb{N}}
\newcommand{\bR}{\mathbb{R}}
\newcommand{\bP}{\mathbb{P}}

\newcommand{\bE}{\mathbb{E}}

\newcommand{\infc}{ \epsilon }

\newcommand{\Ztheta}{u}
\newcommand{\YorZ}{Z}

\newcommand{\Holderindex}{{\rm H}}
\newcommand{\Holderindexl}{{\rm L}}

\newcommand{\gNL}{\mathcal{C} (\infc, f, \Holderindexl )}

\newcommand{\Yp}{a}
\newcommand{\Yq}{b}
\newcommand{\Yr}{c}

\newcommand{\SDEa}{ \sigma }

\newcommand{\bD}{\bR}

\newcommand{\partY}{Y_n }
\newcommand{\partYt}{\partY (t) }
\newcommand{\partSig}{U_n}

\newcommand{\Cantor}{\mathfrak{c}}

\newtheorem{definition}{Definition}[section]
\newtheorem{corollary}{Corollary}[section]
\newtheorem{proposition}{Proposition}[section]
\newtheorem{theorem}{Theorem}[section]
\newtheorem{lemma}{Lemma}[section]
\newtheorem{remark}{Remark}[section]

\newtheorem{example}{Example}[section]

\author{Hiroya Hashimoto}
\address[H.~Hashimoto]{Clinical Research Center, National Hospital Organization Nagoya Medical Center}
\email{hiroyanovs01law@gmail.com}

\author{Takahiro Tsuchiya$^{*}$}
\thanks{$*$ Corresponding author}
\address[T.~Tsuchiya]{School of Computer Science and Engineering, The University of Aizu}
\email[Corresponding author]{suci@probab.com}

\subjclass[2010]{ 
Primary  {60J55}; 
Secondary {41A25}. 
}
\keywords{
Stochastic differential equation,  
Stability problems, 
Convergence rate, 
Local time. 
}

\title{
Stability problems for Cantor stochastic differential equations 
}

\begin{document}

\begin{abstract} 
We consider driftless stochastic differential equations and the diffusions starting from the positive half line. 
It is shown that the Feller test for explosions gives 
a necessary and sufficient condition to hold pathwise uniqueness 
for diffusion coefficients that are positive and monotonically increasing or decreasing on the positive half line 
and the value at the origin is zero. 
Then, stability problems are studied from the aspect of 
H\"older-continuity and a generalized Nakao-Le Gall condition. 
Comparing the convergence rate of H\"older-continuous case, 
the sharpness and stability of the Nakao-Le Gall condition on Cantor stochastic differential equations is confirmed. 
Furthermore, using the Malliavin calculus, 
we construct a smooth solution 
to degenerate second order Fokker-Planck equations under weak conditions on the coefficients. 
\end{abstract}

\maketitle



\section{Introduction}
Given an interval $I=(0,\infty)$ and a real-valued Borel function $\sigma: \bR \to \bR$, 
let us consider 
\begin{align}
X (t) -X (0) = \int_0^t \sigma (X (s)) dB_s, \ X(0) =x_0 \in I. \tag{1}
\end{align}
Skorokhod \cite{SkoroMR0185620} showed that if $\sigma$ is continuous 
that the stochastic differential equations (SDEs) have a weak solution on a filtered probability space.  
Moreover, if 
$\sigma^{-2}$ is integrable over a neighborhood of $x$ for all $x \in I$, 
it follows from Engelbert and Schmidt \cite{Engelbert:1985aa} 
that the weak solution exists law up to stopping time $S$,  
satisfying $\bP (S=\inf \{t \geq 0 : X (t) =0 \})=1$, where the notation $\inf \emptyset$ means $+\infty$.  

Therefore, we are interested in studying the behavior of the solution around the boundary point $0$, 
which has an identity as scale function and speed measure. In fact, we present the following proposition.  
\begin{proposition}\label{proposition iff}
Let $c$ be in $I$ and $\sigma$ be a real-valued Borel function 
that monotonically increases or decreases in $I$ 
such that $(\sigma (x))^2>0$ for every $x \in I$. 
Then, pathwise uniqueness holds for $(\ref{sde0})$ if 
\begin{align*}
\mu (0+) &=\lim_{x \downarrow 0} \int_{x}^{c} \frac{c-y}{(\sigma (y))^2} dy =\infty. 
\end{align*}

Moreover, suppose that $\sigma (0)=0$. 
Then, pathwise uniqueness holds for $(\ref{sde0})$ if and only if $\mu (0+)=\infty.$
\end{proposition}

We first investigate the so-called stability problem proposed by Stroock and Varadhan \cite{SV1979}. 
More precisely, we consider convergence that for Borel measurable functions $\{ \sigma_n \}_{n \in \bN}$, the series of solutions $\{ X_n \}$,  
\begin{align*}
X_n (t) -X_n (0) = \int_0^t \sigma_n(X_{n} (s)) dB_s, \ X_n (0) ={x_0} \in I, 
\end{align*}
converges to $X$ when $\SDEa_n$ tend to $\sigma$ in the following sense, 
\[
\Delta_n := \sup_{x \in \bR} | \sigma (x) - \sigma_n (x) |. 
\]

By Proposition \ref{proposition iff}, 
for the H\"older-continuous diffusion coefficients with exponent $\Holderindex \in [0,1]$,  
pathwise uniqueness holds if and only if $\Holderindex \in [\frac{1}{2},1]$. 
Applying the so-called Yamada-Watanabe method \cite{YW} (cf.~and \cite{Del1998}) to the stability problems 
in Section \ref{Holder continuous diffusion coefficients}, 
we prove the following theorem: 
\begin{theorem}[$\epsilon =0$]\label{theorem for classical stability problems in sup norm}
Let $\Holderindex \in [0, 1]$. 
Suppose that 
$\| \sigma \|_{\infty} = \sup_{x \in \bR} | \sigma (x) | <\infty$ and 
there exists some $c_{\Holderindex}>0$ such that 
for all $x, y \in \bD$ satisfying $|x-y|\leq 1$, $| \sigma (x)- \sigma (y) | \leq  c_{\Holderindex} |x-y|^{\Holderindex}$. 
Moreover, suppose that $X_n$ is a strong solution for all $n \in \mathbb{N}$. 
Then, for any $T>0$, there exists a $C_2 (T,\Holderindex) >0 $ dependent on $\Holderindex$ and $T$ such that 
for every $n \in \bN$ satisfying $\Delta_n <2^{-\Holderindex}$, 
\begin{align*}
 \bE (\sup_{0 \leq t \leq T}|X(t) -X_n (t) |)^2 \leq C_2(T, \Holderindex)  \times 
	\begin{cases}
	\displaystyle  (- \log \Delta_n )^{-\frac{1}{2}} \ & \text{ if } \Holderindex=\frac{1}{2},  \\
	\Delta_n^{ 1- (1/2\Holderindex) }  & \text{ if }  \Holderindex \in (\frac{1}{2},1]. 
	\end{cases}
\end{align*}
\end{theorem}

For Euler-Maruyama's schemes for SDEs with Lipschitz continuous coefficients, 
it is known that Euler-Maruyama's convergence rate is $n^{-\frac{1}{2}}$,  
where the uniform partition is given by $\{ i T /n: i=0,1, \dots, n\} $. 
The distribution of the limit law of the error process 
for Euler-Maruyama's approximation process 
was established through many contributions. 
We refer the reader to \cite[Section 5]{Protterkurtz1991}, \cite{BallyTalay1996}; 
for more details,  see the excellent survey. \cite{Kohatsu2011zbMATH06078912}.  

Recently, 
other efforts have been directed toward extending the this of convergence into {\it stronger}  topologies 
than the one given by {\it weak} convergence of processes,  
as this {\it weak} convergence result 
does not provide full information regarding the rate of convergence of various other functionals. 
Yan in \cite{Yan2002} estimated the bounded rate of Euler-Maruyama's convergence
when $\Holderindex \in (\frac{1}{2},1]$. 
Moreover,  for $\Holderindex = \frac{1}{2}$, 
the log estimation was studied deeply by Gy\"ongy and R\'asonyi in \cite{GR2011}. 
In fact, Theorem \ref{theorem for classical stability problems in sup norm} is 
an analogy of their results for stability problems. 

To the best of our knowledge, 
it is an open problem whether the {\it strong} convergence rates of both 
Euler-Maruyama's schemes and the stability problems under non-Lipschitz coefficients 
are optimal. 
To study the convergence rate under non Lipschitz continuous diffusion coefficients, 
we propose an another approach: 
The key is to consider 
H\"older continuous as a generalized Nakao-Le Gall condition formally defined in Section \ref{Nakao-Le Gall}. 
It induces the following spatial stability problems:  
For $\infc \geq 0$ we introduce
\[
\sigma_n^{\infc} (x) =\infc + \sigma_n (x) \text{ for } x \in I
\]
and 
\begin{align*}
X_n^{\infc} (t) -X_n^{\infc} (0) = \int_0^t \sigma_n^{\infc} (X_n^{\infc} (s)) dB_s, \ X_n^{\infc} (0) ={x_0} \in I. 
\end{align*}
Our main result is as follows: 
\begin{theorem}[$0 < \infc <1$]\label{gNLupper in sup norm}
Suppose that the assumption of Theorem \ref{theorem for classical stability problems in sup norm} holds. 
Furthermore, 
suppose that 
$\sigma$ is non-negative and monotonically increasing or decreasing, 
$\sigma_n$ is 
non-negative in $\bR$ for all $n \in \mathbb{N}$
and $\sup_{n \in \bN}\| \sigma_n \|_{\infty} <\infty$. 
Then, for  any $T>0$, there exists a $T$-dependent $C_4 (T)>0 $ such that 
for every $n \in \bN$ satisfying $\Delta_n <2^{-\frac{1}{2}(\Holderindex+1) }$ and for all $\epsilon \in (0,1)$, 
\begin{align*}
 \bE (\sup_{0 \leq t \leq T }|X^{\infc} (t) -X_n^{\infc} (t)| )^2 \leq \infc^{-3}   C_4 (T) \times 
	\begin{cases}
		\left( -\log  \Delta_n\right)^{-\frac{1}{2}} \ & \text{if } \Holderindex=0,  \\
		\Delta_n^{(1-1/(\Holderindex+1)) } \ &\text{if } \Holderindex \in (0,1] . 
	\end{cases}
\end{align*}
\end{theorem}
As concrete examples, we apply the above results to Cantor diffusion processes $X$, 
\begin{align*}
X(t) -X (0) = \int_0^t \overline{ \Cantor} (X (s) ) dB_s, \ X(0) =x_0 \in (0, 1). 
\end{align*}
where an extended Cantor function satisfies that $\overline{ \Cantor}(x)= \Cantor (x)1_{[0,1]}(x)+1_{(1, \infty)} (x)$ for $x \in \bR$ and 
the middle-$\lambda$ Cantor function $\Cantor$ is obtained 
by removing the middle $\lambda$ of a line segment 
where $\lambda$ is a real number in $(0,1)$. Formally, the middle-$\lambda$ Cantor function 
$\Cantor$ is obtained as the
the limit procedure of a series of bounded, real-valued functions $\{ \Cantor_n \}_{n \in \bN} $ with supremum norms: 
that is, for every $n \in \bN$, $\Delta_n=\sup_{x \in \bR} | \Cantor (x) - \Cantor_n (x) | \leq 2^{-n+2}$. 
These functions are formally defined in Section \ref{Devil's staircases diffusion}. 

If $\lambda \in (0, \frac{1}{2})$, 
$\Cantor$ satisfy H\"older's continuity in Theorem \ref{theorem for classical stability problems in sup norm}, 
thus we obtain that, 
\begin{corollary}\label{Cantor estimation for Holder}
Let $T>0$ and $\lambda \in (0, \frac{1}{2}]$. 
Suppose that $\Cantor_0$ satisfies the condition that 
there exists some $c_{\Holderindex_{\lambda} }>0$ such that 
for all $x, y \in \bD$ satisfying $|x-y|\leq 1$, $| \Cantor_0 (x)-\Cantor_0 (y) | \leq c_{\Holderindex_{\lambda} } |x-y|^{\Holderindex_{\lambda} }$. 
There exists a $C(T, \lambda)>0$ dependent on $T$ and $\lambda$ such that 
\begin{align*}
 \bE (\sup_{0 \leq t \leq T}|X (t) - X_{n}  (t)|) \leq &C(T, \lambda) \times 
	\begin{cases}
	 n^{-\frac{1}{2}}\ &\text{if } \lambda=\frac{1}{2},  \\ 
	2^{-n \left( 1-( 1/2 \Holderindex_\lambda) \right) }\ &\text{if }\lambda \in (0, \frac{1}{2}),  
	\end{cases}
\end{align*}
where $\Holderindex_{\lambda} = { \log{2} }/ { \left( \log{2} - \log (1-\lambda)  \right) } \in [\frac{1}{2}, 1)$. 
\end{corollary}
By Proposition \ref{proposition iff} again, 
the condition $\lambda \in  (\frac{1}{2}, 1)$ implies that pathwise uniqueness of Cantor SDEs does not hold. 
Moreover, $\Cantor_n^{\infc}$ with $\Cantor_0 (x) =0$ for $x \in [0,1]$ and $\infc \in (0,1)$
is not continuous but it is belong to a generalized Nakao-Le Gall condition for every $n \in \bN$. 
Therefore, its convergence rate can be estimated as follows.

\begin{corollary}\label{Cantor estimation for a generalized Nakao-Le Gall}
Let $T>0$. 
There exists  a $T$-dependent $C(T) >0$ such that 
for all $\infc \in (0,1)$ and $\lambda \in (0,1)$. 
\begin{align*}
 \bE (\sup_{0 \leq t \leq T}|X^{\infc} (t) - X_{n}^{\infc} (t)|) \leq 
 	& {\infc}^{-3} C(T) \times 
	2^{-n(1-1/{(\Holderindex_{\lambda}+1)} )} 
\end{align*}
where $\Holderindex_{\lambda} = { \log{2} }/ { \left( \log{2} - \log (1-\lambda)  \right) }\in (0, 1) $. 
\end{corollary}

Finally, 
we rephrase the results in terms of partial differential equation (PDE) analysis. 
More precisely, 
we show the existence of the solutions of Fokker-Planck's equation with degenerate coefficients. 
A typical example is  
\begin{align*}
	\frac{\partial}{\partial t} v(t,x)
		=\frac{1}{2}\frac{\partial^2}{\partial x^2} \left(|x|^{2 \alpha} v(t,x) \right), \ &(t,x) \in (0, T] \times (0, \infty)
\end{align*}
for $\alpha \in [\frac{1}{2},1]$ and $T>0$. 
As the law of $X(t)$ in $(1)$, $\mu_t$ solves the above equation in the sense of distribution. 
Moreover, $X(t)$ is a strong solution, by Proposition \ref{proposition iff}. 
Therefore, 
we construct a smooth fundamental solution $v(t,x)$ such that 
\[
v(t,x)= \int_{[0, \infty)} p(t,x,y)f(y)dy
\]
where $p$ is the density of $\mu_t (dy)=p(t,x,y)dy$ obtained 
by using the Malliavin calculus. 
To the best of our knowledge, 
the coefficients of the above Fokker-Planck equation do not satisfy hypotheses in the PDE literature, 
e.g.~Friedman's book \cite[Chapter 1]{friedman2013partial} for obtaining a weak or smooth solution. 
This is investigated in Section \ref{The Fokker-Planck-Kolmogorov equation}.

The next Section \ref{sec:notation} is devoted to the preliminaries and a useful lemma. 
In Section \ref{Holder continuous diffusion coefficients}, 
the stability problem is applied to the general driftless SDEs with H\"older-continuous diffusion coefficients 
with exponents in $[\frac{1}{2},1]$. 
In Section \ref{Nakao-Le Gall}, we introduce the generalized Nakao-Le Gall condition. 
Then we estimate a strong convergence of stability problems. 
In Section \ref{Devil's staircases diffusion}, we apply the results to Cantor diffusion processes.

%
%
%
%
\section{Preliminaries}\label{sec:notation}
In the interval $I=(0,\infty)$, we assume that a Borel function $\sigma: \bR \to \bR$ satisfying
\begin{align*}
 \text{(LI) }\	\text{if $\sigma (x)\neq 0$, then $\sigma^{-2}$ is integrable over a neighborhood of }  x. 
\end{align*}
As a consequence of \cite[Theorem 1]{Engelbert:1985ab}, 
we have a weak solution $X$ on the filtered probability space $(\Omega, \mathscr{F}, \mathscr{F}_t, \bP)$ 
such that 
\begin{align}\label{sde0}
X (t) -X (0) = \int_0^t \sigma (X (s)) dB_s, \ X(0) =x_0 \in I. 
\end{align}
We also define a stopping time $S= \inf \{t \geq 0 : X (t) \notin I \}$. 
Because $X$ satisfies a driftless SDEs, its scale function is 
the identity (see \cite[5.22 Proposition]{Karatzasbook}); 
\[
\bP ( \lim_{t \uparrow S} X (t) =0 ) =\bP ( \sup_{0 \leq t< S} X (t) <\infty )=1.  
\]
Therefore the solution is nonexplosive, 
and we have 
\begin{align}\label{a stopping time}
\bP \left( S=\inf \{t \geq 0 : X (t) =0 \} \right)=1. 
\end{align}
As pathwise uniqueness implies uniqueness in the sense of probability law by the Yamada-Watanabe theorem \cite{YW},  
the following lemma obviously follows from Engelbert and Schmidt \cite{Engelbert:1985ab}. 
However, we provide a proof here. 

\begin{lemma}\label{Lemma: only if}
Let $c$ be in $I$. 
Suppose there exists a solution to equation $(\ref{sde0})$ and $\sigma (0)=0$. 
If pathwise uniqueness holds for $(\ref{sde0})$, 
then it follows that 
\[
\lim_{x \downarrow 0} 
\int_{x}^{c} \frac{1}{(\sigma (y))^2} dy = \infty.
\]
\end{lemma}
\begin{proof}
We shall prove that 
if $\displaystyle \lim_{x \downarrow 0} \int_{x}^{c} \frac{1}{(\sigma (y))^2} dy < \infty$, 
then the pathwise uniqueness is violated. 

By the Feller's test, because
\begin{align*}
\nu (0+) &=\lim_{x \downarrow 0} 
\int_{x}^{c} \frac{y-x}{(\sigma (y))^2} dy 
\leq 2 c \lim_{x \downarrow 0} \int_{x}^{c} \frac{1}{(\sigma (y))^2} dy <\infty , 
\\ \mu (0+) &=\lim_{x \downarrow 0} 
\int_{x}^{c} \frac{c-y}{(\sigma (y))^2} dy 
\leq 2 c \lim_{x \downarrow 0} \int_{x}^{c} \frac{1}{(\sigma (y))^2} dy <\infty, 
\end{align*}
the boundary point $0$ is exit and entrance. 
Moreover, considering the natural scale $\lim_{x \uparrow \infty}  s (x)=\lim_{x \uparrow \infty} (x-c) =\infty$, 
$0$ is reached from an interior point of $I$ in finite time $\bP$-almost surely; that is 
\[
	\bP (S=\inf \{t \geq 0 : X (t) =0 \} <\infty)=1.
\]
Applying the method of time change, 
let us now consider a diffusion process starting from $0$. 
For a one-dimensional $\overline{\mathscr{F}}_t$-Brownian motion $\overline{B}_t$, let
\[
	A_t := \int_0^t \frac{1}{\sigma ( |\overline{B}_s|)^{2}} ds. 
\]
Assuming that $\displaystyle \lim_{x \downarrow 0} \int_{x}^{c} \frac{1}{(\sigma (y))^2} dy < \infty$, 
then $\sigma^{-2}(|\cdot |)$ is locally integrable over $\bR$. 
Thus, it follows that $\bP ( A_t < \infty)=1$ for every $t \geq 0$. 
Moreover, 
let $A_t$ define a continuous additive function of $\overline{B}_t$ such that 
\[
 \bE ( |\overline{B}_{A^{-1}_t}|^2)=\int_{0}^{t} \sigma ( \overline{B}_{A^{-1}_s} )^2 ds. 
 \]
Therefore, for every $t \geq 0$, we have 
\[
	B_t := \int_{0}^{t} \frac{1}{\sigma ( X_s )} d X_s \text{ where }  X_t = \overline{B}_{A^{-1}_s}. 
\]
Thus, $(X_t , B_t)$ satisfies equation $(\ref{sde0})$ with $X (0)=0$ for $t\geq S$. However, 
because $\sigma (0)=0$, $(0 , B_t)$ also satisfies this equation, violating the pathwise uniqueness condition. 
\end{proof}
Conversely, 
as is well-known, the existence of a weak and unique solution in the probability law does not 
imply pathwise uniqueness without some additional condition. 
On the basis of the above proof, we impose an additional condition of the behavior of the boundary $0$ as follows. 
\begin{lemma}\label{Lemma: only if2}
Let $c$ be in $I$. 
Suppose there exists a solution to equation $(\ref{sde0})$ and let $\sigma (0)=0$. 
If pathwise uniqueness holds for $(\ref{sde0})$, 
then 
\[
\mu (0+) =\lim_{x \downarrow 0} \int_{x}^{c} \frac{c-y}{(\sigma (y))^2} dy =\infty. 
\]
\end{lemma}

\begin{proof}
Following the proof of Lemma \ref{Lemma: only if}, we have 
\[
\mu (0+) =\lim_{x \downarrow 0} \int_{x}^{c} \frac{c-y}{(\sigma (y))^2} dy =\infty \Rightarrow 
\int_{x}^{c} \frac{1}{(\sigma (y))^2} dy = \infty. 
\]
Thus, we obtain the desired result. 
\end{proof}

\begin{remark}\label{remark: Feller}
By the Feller test (see \cite[Chapter 6]{handbookzbMATH01817636}), for $c\in (0,1)$ and $x \in (0, c)\subset I$, we define 
	\begin{align*}
	\nu (x) = \int_{x}^{c} \frac{y-x}{(\sigma (y))^2} dy \text{ and } \mu (x) =\int_{x}^{c} \frac{c-y}{(\sigma (y))^2} dy.
	\end{align*}
\begin{enumerate}
\item If $0$ is natural; that is 
	\begin{align*}
	\nu (0+) =\lim_{x \downarrow 0} \nu(x) = \infty \text{ and } \mu (0+) =\lim_{x \downarrow 0} \int_{x}^{c}\mu (x) =\infty,
	\end{align*}
	then we have $\bP (S<\infty)=0$ and $\bP ( X (t) \in I  \text{ for }  t\geq 0 )=1$. 
\item If $0$ is an exit and no-entry point, i.e., $\nu (0+) <\infty$ and $\mu (0+)= \infty$, 
	then $0<\bP (S< \infty)\leq1$ and 
	\[
	\bP (X_t = 0 \text{ for every } t \geq S)>0.
	\]
	However, this process cannot be started from the boundary point.  
\item If $\nu (0+) =\infty$ and $\mu (0+)< \infty$, then it implies 
	\[
	\infty=\lim_{x \downarrow 0} \int_{x}^{c} \frac{y}{(\sigma (y))^2} dy \leq \lim_{x \downarrow 0} \int_{x}^{c} \frac{1}{(\sigma (y))^2} dy, 
	\]
	is implied. Moreover, 
	\[
	 \lim_{x \downarrow 0} \frac{ \displaystyle \int_{x}^{c} \frac{1}{(\sigma (y))^2}dy }{ \displaystyle \int_{x}^{c} \frac{y}{(\sigma (y))^2}dy }= \lim_{x \downarrow 0}\frac{1}{x}=\infty, 
	\]
	which is inconsistent with $\mu (0+)< \infty$. Therefore, under this condition, 
	$0$ is not an entry and nonexit point 
\item  If $0$ is both an exit and an entry point, i.e., 
	$\nu (0+) <\infty$ and $\mu (0+)< \infty$, it is called a nonsingular point. 
	As the scale function $\lim_{x \uparrow \infty}  s (x)=\infty$, we conclude that $\bP (S< \infty)=1$. 
	Moreover, diffusion can begin from a nonsingular boundary. 
\end{enumerate}
\end{remark}

\begin{lemma}\label{lemma if}
Let $c$ be in $I$ and $\sigma$ be a real-valued Borel function that monotonically increase or decrease in $I$ 
such that $(\sigma (x))^2>0$ for every $x \in I$. 
%
Then, pathwise uniqueness holds for $(\ref{sde0})$ if 
\begin{align*}
\mu (0+) &=\lim_{x \downarrow 0} 
\int_{x}^{c} \frac{c-y}{(\sigma (y))^2} dy =\infty. 
\end{align*}
\end{lemma}

\begin{example}
According to this expression, $\mu (0+) <\infty$ generally does not imply pathwise uniqueness. 
To show this, we consider Tanaka's example, in which 
$\sigma(x)=1_{\{ x>0\}} (x) -1_{\{ x\leq 0\}} (x) $ for $x \in \bR$ is monotone increasing but 
$\mu (0+) =c-\frac{1}{2}c^2 <\infty$. 
\end{example}

\begin{example}
Suppose that $\sigma$ is a H\"older-continuous function with exponents $\alpha \in [0,1)$. 
In Girsanov's example \cite{GirsanovzbMATH03196595}, 
$\sigma(x)=\frac{|x|^\alpha}{1+|x|^\alpha} $ for $x \in \bR$ for $0 \leq \alpha < \frac{1}{2}$ implies that $\mu (0+) <\infty$.
Moreover, the Yamada-Watanabe theorem \cite[Theorem 1]{YW}, 
gives $\mu (0+) =\infty$ if $\sigma$ is H\"older-continuous of order $\frac{1}{2}$, i.e., $\alpha \geq \frac{1}{2}$. 
\end{example}

\begin{remark}
Bass, Burdzy and Chen \cite{Bass07AOP} recently showed that 
if the solution is instantaneous at $0$, i.e., $\int_0^\infty 1_{ \{ 0\}} (X(s)) ds=0$, where $\sigma(x)=|x|^\alpha$ for $0 < \alpha < \frac{1}{2}$, then pathwise uniqueness holds. 
\end{remark}

\begin{proof}[Proof of Lemma \ref{lemma if}]
\begin{subequations}
Denoting the solutions of the SDE $(\ref{sde0})$ by $X$, 
for every $t \geq 0$, we have 
\begin{align}\label{sde0+ln}
X (t) -X (0) = \int_0^t \sigma (X (s)) dB_s, \ X(0) =x_0. 
\end{align}
It is sufficient to prove that pathwise uniqueness holds up to time $S (\leq \infty)$, defined as $(\ref{a stopping time})$. 

Now we define 
\begin{align*}
I_n =(l_n, \infty) \text{ and }  
l_n := \inf\left\{ x \in I : \sigma (x) > \frac{1}{n} \right\}. 
\end{align*}
Under the assumption $(\sigma (x))^2>0$ for every $x \in I$, 
there exists a $n_0 \in \bN$ such that 
$X(0)=x_0 \in I_n$ for every $n \geq n_0 \in \bN$. 
Set 
\begin{align*}
S_n =  \inf\left\{ t \geq 0  :  X(t) \notin  I_n \right\}. 
\end{align*}
Then, as $l_n$ is a nonsingular point, we can write 
\[
\bP ( S_n < \infty )=1.
\]
For $n \in \bN$, we consider the following SDE 
for every $t\geq 0$, 
\begin{align}\label{sde0+ln1}
X_n (t) -X_n (0) = \int_0^t \sigma_n (X_n (s)) dB_s, \ X_n (0) =x_0. 
\end{align}
where $\sigma_n$ satisfies that for every $x \in \bR$, $\sigma_n (x) = \sigma (-x)$ and 
\[
 \sigma_n (x) :=\sigma (x) 1_{(l_n, \infty)} (x)  +\frac{1}{n} 1_{[0, l_n ]} (x) . 
\]
Here $\sigma_n$ is a Borel function such that $ \sigma_n (x)  \geq \frac{1}{n} $ for $x \in I$. 
Therefore, $ \sigma_n^{-2}$ is integrable over a neighborhood of every $x \in \bR$ , 
then there exists a unique solution in the sense of the probability law, see \cite{Engelbert:1985aa}. 
Moreover, because $\sigma$ is monotonically increasing or decreasing in $I$, 
the variation of $\sigma_n$ is finite, see \cite[(4.2) Proposition]{RevuzYor1999}. 
Moreover, $ \sigma_n (x)  \geq \frac{1}{n} $ for $x \in I$. 
Thus, by Nakao's result \cite{N1972},  pathwise uniqueness holds for $(\ref{sde0+ln1})$. 
We conclude that $X_n$ is unique and strong solution by the Yamada-Watanabe theorem \cite{YW}. 

Let us now consider a strong solution $X_n^S$ starting from $l_n$, 
\begin{align}\label{sde0+ln1+ln}
X_n^{S} (u) -X_n^{S} (0) = \int_0^u \sigma_n (X_n^S (s)) dW_s, \ X_n^S (0) =l_n. 
\end{align}
For an almost surely finite stopping time $S_n$, 
we define a Brownian motion independent of $\mathscr{F}_{S_n}$ such that 
$W_s := B_{s+S_n}-B_{S_n}$ for $s \geq 0$, see \cite[(6.15) Theorem in page 86]{Karatzasbook}. 

For every $t \geq S_n$, 
$W_{t-S_n} = B(t)-B(S_n)$ is $\mathscr{F}_{t}$-measurable. 
It follows that $\mathscr{F}_{t-S_n}^W \subset \mathscr{F}_{t}$ for every $t \geq S_n$. 
Furthermore, 
since $X_n^S$ is a strong solution with respect to $\mathscr{F}^{W}$, 
$X_n^{S} (t-S_n)$ of $(\ref{sde0+ln1+ln})$ is $\mathscr{F}_{t}^W$-measurable. 
We then conclude that 
$X_n^{S} (t-S_n)$ of $(\ref{sde0+ln1+ln})$ is also $\mathscr{F}_{t}$-measurable. 

Therefore, 
defining 
\[
Z_n (t) := X(t) 1_{\{ t < S_n \}} + X_n^S (t-S_n) 1_{\{ t \geq S_n \}}, 
\]
we find that $Z_n$ is a solution of $(\ref{sde0+ln1})$. 

On the other hand, again noting that pathwise uniqueness holds for $(\ref{sde0+ln1})$, 
if $X$ and $\overline{X}$ be solutions of the SDE $(\ref{sde0})$, then 
\[
\bP ( Z_n (t) = \overline{Z}_n (t)  \text{ for every } t \geq 0 )=1, 
\]
similarly, we define
$\overline{Z}_n:=\overline{X}(t) 1_{\{ t < \overline{S}_n \}} + X_n^{\overline{S}} (t-\overline{S}_n) 1_{\{ t \geq \overline{S}_n \}}$ 
and $\overline{S}_n = \inf\left\{ t \geq 0 :  \overline{X} (t) \notin I_n \right\}$. 
Therefore, 
we have 
\[
\bP ( X(t) = \overline{X}(t)  \text{ for every } t \leq S_n \wedge \overline{S}_n )=1.
\]
Moreover, since $S(\omega) =\lim_{n \to \infty}S_n(\omega) =\lim_{n \to \infty}\overline{S}_n(\omega)$ for almost every $\omega \in \Omega$, 
we conclude that 
\[
\bP ( X(t) = \overline{X}(t)  \text{ for every } t \leq S )=1, 
\] 
which is the desired result. 
\end{subequations}
\end{proof}

Therefore, we obtain Proposition \ref{proposition iff}. 
\begin{proof}[Proof of Proposition \ref{proposition iff}]
By Lemma \ref{Lemma: only if} and \ref{lemma if}, we obtain the result. 
\end{proof}

To investigate stability problems, 
let $X_n$ be the solution to the following equation, 
where $\sigma_n$ are Borel functions with $n \in \bN$: 
\begin{align}\label{sde1n}
X_n (t) -  X_n (0) = \int_{0}^t \SDEa_n ( X_n (s) ) dB_s, \ X_n (0) =x_0 \in I  \text{ and } \sigma_n (0) =\sigma (0) . 
\end{align}
For every $t \geq 0$, we write 
\begin{align*}
\partYt = X (t) - X_n (t) . 
\end{align*} 
Recall that for any $n \in \bN$, $\partY$ is a continuous martingale 
with a quadratic variation of $\langle \partY \rangle$. 
We also define $L_t^x (\partY)$, the continuous version of the local time 
$L_t^x (\partY )$ of $\partY$ accumulated at $x$ until time $t$  (cf.~\cite[page 391, (3.5) Theorem]{RevuzYor1999}). 

The following lemma is useful for estimating 
the convergence rate. 
%
%
\begin{lemma}\label{main claim}
Let $M$ be a local continuous martingale with $\bP (M(0)=0)=1$
and let $\rho$ be a monotonic, strictly increase function on $[0, 1]$ with $\rho (0)=0$. 
For any sequences $\{ \Yp_n \}_{n \in \bN} , \ \{\Yq_n \}_{n \in \bN}$ satisfying $0< \Yp_n \leq \Yq_n \leq 1 \ (n \in \bN)$, 
we define a positive sequence $\{ c_n \}_{n \in \bN}$, 
\[
\Yr_n := \int_{\Yp_n}^{\Yq_n}  \frac{dy}{\rho (y)}. 
\]
Then, 
for every $t \geq 0$ and $n \in \bN$, we have 
\begin{align*}
\bE (|M(t) |) -\Yq_n
	\leq \bE \left(\frac{1}{2 \Yr_n} \int_0^{t} \frac{d \langle M \rangle_s }{\rho (|M (s) |)} 1_{ (\Yp_n, \Yq_n)} (|M (s)|)   \right)
	\leq \bE (|M(t) |).
\end{align*}
\end{lemma}
\begin{proof}
Let 
$\varphi_{n} $ be a symmetric function around the origin such that 
for all $x\in \bR$, 
\begin{align*}
\varphi_{n} (x) &= 1_{[\Yp_n, \infty) } (x)\int_{\Yp_n}^{|x|}
	\int_{\Yp_n}^{y}  \frac{1}{c_n \rho (z)  } 1_{ (\Yp_n, \Yq_n)} (z) dz\, dy.
\end{align*}
%
Thus, for all $x \in \bR$ with $|x| > \Yp_n$, we have 
\begin{align*}
\varphi_{n} (x) &\leq  \int_{\Yp_n}^{|x|}  dy = |x| -  \Yp_n  < | x|, 
\end{align*}
and for every $x \in \bR$ such that $|x| > \Yq_n$, we have 
\[
\varphi_{n} (x) \geq  \int_{\Yq_n}^{|x|}  \int_{\Yp_n}^{y}  \frac{1}{c_n \rho (z)  } 1_{ (\Yp_n, \Yq_n)} (z) dz dy =|x| - \Yq_n. 
\]
That is, for all $x \in \bR$
\begin{align*}
& |x| - \Yq_n \leq \varphi_{n}  (x) \leq |x|.
\end{align*}
Moreover, $\varphi'_{n}$ and $\varphi''_{n}$ are Lebesgue-almost for every $x \in \bR$.  Thus, we can write 
\begin{align*}
		| \varphi'_{n}   (x) | \leq 1  \text{ and }  
		\varphi''_{n} (x)  = \frac{1}{c_n \rho_n (|x|)} 1_{(\Yp_n, \Yq_n) } (|x|) .
\end{align*}
Applying 
It\^o's  formula to $\varphi_{n} \left(M ( \cdot ) \right) $ (see \cite[7.3 Problem]{Karatzasbook}), 
we obtain the $\bP$-almost surely as 
\begin{align*}
\varphi_{n}  (M (t))   - \varphi_{n}  (M (0))  
&=\int_0^t \varphi'_{n}  (M (s)) dM_s  + \frac{1}{2} \int_0^t \varphi''_{n}  (M (s)) d \langle M \rangle_s  
\\&= \int_0^t \varphi'_{n}  (M (s)) dM_s+ \frac{1}{2\Yr_n} \int_0^{t} \frac{d \langle M \rangle_s }{\rho (|M (s) |)} 1_{ (\Yp_n, \Yq_n)}(|M (s)|).  
\end{align*}
As $ \int_0^t \varphi'_{n}  (M (s)) dM (s) $ is a martingale vanishing at $0$ for every $n \in \bN$, 
\begin{align*}
\bE \left( \varphi_{n}  (M (t))   - \varphi_{n}  (M (0))  \right)
= \frac{1}{2 \Yr_n}\bE \left( \int_0^{t} \frac{d \langle M \rangle_s }{\rho (|M (s) |)} 1_{ (\Yp_n, \Yq_n)} (|M (s)|)
 	\right). 
\end{align*}
Moreover, 
$|\varphi'_{n}(x) | \leq 1 $ for $x \in \bR$, 
$|M (t) | - \Yq_n \leq \varphi_{n}  (M (t)) \leq |M (t)|$ and $\varphi_{n}  (M (0)) = 0$ are $\bP$-almost surely for every $t \geq 0$ and $n \in \bN$. 
Thus, we obtain the desired result. 
\end{proof}

\section{H\"older-Continuous Diffusion Coefficients}\label{Holder continuous diffusion coefficients}
In this section, 
we consider stability problems in the case H\"older-continuous diffusion coefficients. 
Given real-valued Borel functions $\sigma$ and $\sigma_n$ with $n \in \bN$, 
let $X$ be a solution of $(\ref{sde0})$ and $X_n$ be a solution of $(\ref{sde1n})$. 
For every $n \in \bN$, we define 
\begin{align}\label{def:delta}
 \Delta_n := \sup_{x \in \bD} | \sigma (x) - \sigma_n (x) | < \infty. 
\end{align}
The convergence rate is obtained as follows cf.~\cite{KY}. 
\begin{lemma}
\label{lemma for classical stability problems}
Suppose that the assumption of Theorem \ref{theorem for classical stability problems in sup norm} holds. 
Then, there exists a $C_1 (t,\Holderindex)>0$ dependent on $\Holderindex$ and $t$ such that 
for every $t \geq 0$ and $n \in \bN$ satisfying $\Delta_n <2^{-\Holderindex}$, 
\begin{align*}
\bE (| X(t) -X_n (t)  |)\leq C_1(t, \Holderindex) \times
	\begin{cases}
	\displaystyle  \left( - \log  \Delta_n  \right)^{-1} \ & \text{ if } \Holderindex=\frac{1}{2},  \\
	\displaystyle  \Delta_n^{2 -(1/\Holderindex) }
	& \text{ if }  \Holderindex \in (\frac{1}{2},1]. 
	\end{cases}
\end{align*}
\end{lemma}
\begin{remark}
Here,  $C_1(t,\Holderindex)  \leq  \frac{2}{\log 2}  {\left(1 + t c_{\Holderindex}^{2} +t  \right) }$.
\end{remark}

\begin{proof}
Recall that $Y_n$ is defined by 
\[
Y_n (t) =X(t) - X_n (t) = \int_{0}^t \sigma (X(s) ) - \sigma_n (X_n (s) ) dB_s. 
\]
Thus, 
for every $t \geq 0$ and $n \in \bN$ 
we define $U_n$ and $V_n$ as follows: 
\begin{align*}
	U_n (t) &=\int_{0}^t \sigma (X(s) ) - \sigma (X_n (s) ) dB_s \text{ and }
	V_n (t) =\int_{0}^t \sigma (X_n (s) ) - \SDEa_n (X_n (s) ) dB_s. 
\end{align*}
By Kunita and Watanabe's inequality \cite[(1.15) Proposition in Capter IV]{RevuzYor1999} and Jensen's inequality, 
for all $s \geq 0$ and $ n \in \bN$, we have 
\begin{align*}
\langle  Y_n  \rangle_s 
	=\langle  U_n  + V_n   \rangle_s 
	&\leq \langle U_n \rangle_s  +2\langle U_n \rangle_s^{\frac{1}{2}} \langle  V_n \rangle_s^{\frac{1}{2}} +\langle V_n \rangle_s 
 	\\&\leq 2\langle U_n \rangle_s  +2\langle  V_n \rangle_s. 
\end{align*}
By the H\"older continuity of $\sigma$ and the definition of $\Delta_n$ in $(\ref{def:delta})$, we have 
\[
\langle U_n \rangle_s \leq c_{\Holderindex}^{2} |X(s)- X_n (s)|^{2\Holderindex} \text{ and } 
\langle V_n \rangle_s \leq \Delta_n^2.  
\]
It follows that 
\begin{align*}
&\int_0^{t} \frac{d\langle  Y_n   \rangle_s   }{\rho (|\partY (s) |)} 1_{ (\Yp_n, \Yq_n)} (|\partY (s)|) 
	\leq 	2t \left( c_{\Holderindex}^{2}	+ \left(\frac{\Delta_n}{\Yp_n^{\Holderindex}}\right)^2  \right). 
\end{align*}
where we define $\rho  (u) =u^{2\Holderindex} \ (u \geq 0)$. 
By Lemma \ref{main claim}, we have 
\begin{subequations}
\begin{align}\label{Yn estimation0}
\bE (| X(t) -X_n (t)  |)
\leq \Yq_n +\frac{t}{ \Yr_n} 	 \left( c_{\Holderindex}^{2}	+ \left(\frac{\Delta_n}{\Yp_n^{\Holderindex}}\right)^2  \right).
\end{align}

For $ 0<\Yp_n < \Yq_n <1$ where $n \in \bN$, 
the mean-value theorem gives 
\[
0< \left( \Yq_n - \Yp_n \right) \leq \Yr_n =\int_{\Yp_n}^{\Yq_n}  \frac{dx}{x^{2\Holderindex}} \leq \frac{1}{\Yp_n } \left( \Yq_n - \Yp_n \right).
\]
It follows that $\Yq_n\Yr_n  \leq \Yp_n^{-1} \Yq_n\left( \Yq_n - \Yp_n \right)<1$, 
that is $\Yq_n \leq 1/\Yr_n$ for every $n \in \bN$, 
and 
\[
\bE (| X(t) -X_n (t)  |)\leq \frac{1}{ \Yr_n} \left(1 + t c_{\Holderindex}^2 +
	 t \left(\frac{\Delta_n}{\Yp_n^{\Holderindex}}\right)^2   \right) .
\]
For $\Holderindex  \in [\frac{1}{2}, 1]$, we have 
\[
\Yr_n = \Yp_n^{1-2\Holderindex} \int_{1}^{\Yp_n^{-1} \Yq_n}  \frac{dy}{y^{2\Holderindex}}. 
\]
If $\Holderindex =\frac{1}{2}$, we select $\Yq_n = \sqrt{\Yp_n}$ and hence obtain 
\[
	\Yr_n = -\frac{1}{2} \log \Yp_n.
\] 
If $\Holderindex \in (\frac{1}{2},1]$, for every $n \in \bN$ satisfying $\Yq_n =2 \Yp_n <1$, we obtain 
\[
	\Yr_n = \Yp_n^{1-2\Holderindex} \left\{\frac{1}{1-2\Holderindex} (2^{1-2\Holderindex}-1)\right\} \geq \Yp_n^{1-2\Holderindex} \times  \frac{1}{2} \log 2.
\]
Finally, putting $\Yp_n =\Delta_n^{1/\Holderindex} $, 
for every $n \in \bN$ satisfying $\Delta_n <2^{-\Holderindex}$, we have $0 \leq \Yp_n \leq \Yq_n \leq 1$. 
We hence conclude that 
for every $n \in \bN$ satisfying $\Delta_n <2^{-\Holderindex}$ and $t \geq 0$, we have 
\[
\bE (| X(t) -X_n (t)  |)\leq \left(1 + t c_{\Holderindex}^2 +t  \right) 
	\begin{cases}
	\displaystyle 2 \left( - \log  \Delta_n  \right)^{-1} \ & \text{ if } \Holderindex=\frac{1}{2},  \\
	\displaystyle \frac{2}{\log 2} \Delta_n^{ 2 -(1/\Holderindex) }
	& \text{ if }  \Holderindex \in (\frac{1}{2},1]. 
	\end{cases}
\]
\end{subequations}
\end{proof}

By the $L^1$-estimation, we obtain a convergence rate in terms of the $L^2$-sup-norm. 
\begin{proof}[Proof of Theorem \ref{theorem for classical stability problems in sup norm}]
Let $r$ be an arbitrary positive number. 
For every $s \geq 0$ and $n \in \bN$, we have $\bP$-almost surely, 
\begin{align*}
	&\left\{ \sigma (X(s) ) - \sigma_n (X_n (s) )  \right\}^2
	=\left\{ \sigma (X(s) ) - \sigma (X_n (s) ) +\sigma (X_n (s) ) - \sigma_n (X_n (s) )  \right\}^2 
	\\&\leq 2\left\{ \sigma (X(s) ) - \sigma (X_n (s) )  \right\}^2 +2\Delta_n^2
	\\&= 2\left\{ \sigma (X(s) ) - \sigma (X_n (s) )  \right\}^2 
		\left( 1_{\left\{ |X(s)- X_n (s)| <r \right\} } + 1_{\left\{ |X(s)- X_n (s)| \geq r \right\} } \right)
		+2\Delta_n^2
	\\&\leq 2 c_{\Holderindex}^{2} r^{2\Holderindex}+4\| \sigma \|_{\infty}^2 1_{\left\{ |X(s)- X_n (s)| \geq r \right\} }
		+2 \Delta_n^2.
\end{align*}
By Chebyshev's inequality, we obtain 
\begin{align*}
\bE ( \int_{0}^T \left\{ \sigma (X(s) ) - \sigma_n (X_n (s) )  \right\}^2 ds )
 	\leq \int_0^T 2c_{\Holderindex}^{2} r^{2\Holderindex}+ 4\| \sigma \|_{\infty}^2  r^{-1} \bE (|X(t)-X_n (t)|) +2 \Delta_n^2 \, ds. 
\end{align*}
Thus, selecting $r$ such that 
\[
r =\left(  \sup_{0 \leq t \leq T} \bE (|X(t)-X_n (t)|)  \right)^{\frac{1}{2}},
\]
since $\Delta_n^2 \leq r^{2\Holderindex} \leq r^{\frac{1}{2} }$ for $r \in [0,1]$ and $\Holderindex \in [\frac{1}{2}, 1]$, we obtain 
\begin{align*}
\bE ( \int_{0}^T \left\{ \sigma (X(s) ) - \sigma_n (X_n (s) )  \right\}^2 ds )
 	\leq ( 2c_{\Holderindex}^{2}+ 4\| \sigma \|_{\infty}^2+2)T \left(  \sup_{0 \leq t \leq T} \bE (|X(t)-X_n (t)|)  \right)^{\frac{1}{2}}. 
\end{align*}
Doob's maximal inequality then gives 
\begin{align*}
 \bE (\sup_{0 \leq t \leq T}|X(t) -X_n (t) |)^2 
 	\leq 4 ( 2c_{\Holderindex}^{2}+ 4\| \sigma \|_{\infty}^2+2)T \left(  \sup_{0 \leq t \leq T} \bE (|X(t)-X_n (t)|)  \right)^{\frac{1}{2}}. 
\end{align*}
The desired inequality then follows  from Lemma \ref{lemma for classical stability problems}. 
\end{proof}
\begin{remark}
$C_2 (T, \Holderindex)$ satisfies $C_2 (T,\Holderindex) \leq 4 ( 2c_{\Holderindex}^{2}+ 4\| \sigma \|_{\infty}^2+2)T \sqrt{ C_1 (T,\Holderindex)}$
and $C_1 (T,\Holderindex)$ is given by Lemma \ref{lemma for classical stability problems}. 
\end{remark}

\section{A Generalization of the Nakao-Le Gall Condition}\label{Nakao-Le Gall}
In this section, 
we consider nondegenerate diffusion coefficients satisfying a generalized Nakao-Le Gall condition. 
This treatment is slightly modified from that of Le Gall \cite{LeGall1983} and Nakao \cite{N1972} . 
\begin{definition}[Generalized Nakao-Le Gall condition]\label{conditionC}
Let $\infc$ and $\Holderindexl$ be in $[0, \frac{1}{2}]$ and $f$ be 
a monotonically increasing and uniformly bounded function such that $\| f \|_{\infty} :=\sup_{x \in \bR} |f(x)| < \infty $. 
If a Borel function $\sigma$ satisfies
\begin{align*}
 \infc \leq \sigma (x) 
\end{align*}
for every $x \in \bR$ and 
\begin{align*}
|\sigma (x) - \sigma (y) |  \leq  |x-y|^{\Holderindexl}| f(x) - f(y) |^{\frac{1}{2}},
\end{align*}
for every $x, y \in \bD$ and $|x-y|\leq 1$, 
then $\sigma$ satisfies the generalized Nakao-Le Gall condition, and we write $ \sigma \in \gNL$. 
\end{definition}

\begin{remark}\label{gNL has a unique strong solution}
Suppose that $ \sigma \in \mathcal{C} (\infc,  f, 0 )$ and $\infc>0$. 
Then 
$\sigma^{-2}$ is locally integrable, 
and equation $(\ref{sde0})$ has non-trivial weak solution up to stopping time $S$ given by $(\ref{a stopping time})$. 
By the Feller test, the boundary $0$ is a nonsingular boundary, and $\bP (S<\infty)=1$. 
Furthermore, Le Gall showed that the local time of boundary $0$ is zero: 
\[
L_t^{0}(X) =0  \text{ for every } t \geq 0. 
\]
Thus, pathwise uniqueness holds for $(\ref{sde0})$. 
Because $\mathcal{C} (\infc, f, \Holderindexl ) \subset \mathcal{C} (\infc,  f, 0 )$, 
we conclude that 
the SDEs $(\ref{sde0})$ with $ \sigma \in \gNL$ have unique strong solutions if $\infc>0$. 
\end{remark}

Let $\sigma$ be a real-valued Borel function $\sigma: \bR \to \bR$ again satisfies 
$\left( \sigma (x) \right)^2 >0$ for every $x \in I$. 
A solution $X$ of the SDE $(\ref{sde0})$ satisfies 
\begin{align*}
X (t) -X (0) = \int_0^t \sigma (X (s)) dB_s, \ X(0) =x_0 \in I. 
\end{align*}
For every $\infc \geq 0$ and $x \in \bR$, we denote 
\[
\sigma^{\infc} (x) := \infc + \sigma (x), 
\]
and write a solution $X^\infc$ as follows: 
\begin{align*}
X^\infc (t) -X^\infc (0) = \int_0^t \sigma^\infc (X^\infc (s)) dB_s, \ X^\infc(0) =x_0 \in I. 
\end{align*} 
In the same way, we define $\sigma^{\infc}_n$ and $X_n^\infc$ for solutions $X_n$ of the SDE $(\ref{sde1n})$. 
\begin{example}\label{positive sigma and PU}
If $\sigma$ is H\"older-continuous with exponent $\Holderindex \in [0,1)$, it follows that for every $x,y \in \bR$, 
\[
\left| \sigma(x) -  \sigma(y) \right| \leq 	|x-y|^{\frac{\Holderindex}{2}} 	\left| \sigma(x) -  \sigma(y) \right|^{\frac{1}{2} }. 
\]
Moreover, suppose that $\sigma$ a monotonically increasing or decreasing non-negative bounded function. 
Then $\sigma^\infc \in  \mathcal{C} (\infc,  \sigma, \frac{\Holderindex}{2} )$. 
\end{example}
\begin{example}
Let us consider the skewed Brownian diffusion coefficient \cite{HarrisonShepp} given by 
$\sigma_{\alpha} (x) = \alpha 1_{\left\{x>0 \right\} }(x)+ (1-\alpha) 1_{\left\{ x \leq 0 \right\} }(x)$, 
where $x \in \bR$. 
Then we have $\sigma_{\alpha} \in \mathcal{C} (\alpha,  \sigma_\alpha, 0)$.
\end{example}

For analyzing stability problem, we recall 
\[
\Delta_n = \sup_{x \in \bR} |\sigma (x) - \sigma_n (x)|. 
\]
\begin{lemma}\label{gNLupper}
Suppose that $\sigma \in \mathcal{C} (0, f, \Holderindexl )$ and 
$\{ \sigma_n \}_{n \in \bN}$ is a series of Borel functions satisfying 
$\sigma_n (x) \geq 0 $ for $\bR$ and $\sup_{n \in \bN}\| \sigma_n\|_\infty < \infty$. 
Then, there exists a $C_3 (t) >0$ dependent on $t \geq 0$ 
such that for every $n \in \bN$ satisfying $\Delta_n<2^{-(\Holderindexl+\frac{1}{2}) }$ and for every $\infc \in (0,1)$, 
\begin{align*}
 \bE (|X^\infc (t) -X_n^{\infc} (t)  |) \
 	\leq & {\infc^{-2}}  C_3 (t) \times 
	\begin{cases}
		\left( -\log  \Delta_n\right)^{-1} \ & \text{if } \Holderindexl=0,  \\
		\Delta_n^{2-(2/(2\Holderindexl+1))} \ &\text{if } \Holderindexl \in (0,\frac{1}{2}].
	\end{cases}
\end{align*}
\end{lemma}
\begin{remark}
For any $\Holderindexl \in [0,\frac{1}{2}]$, $C_3$ satisfies 
\[
C_3 (t) \leq 
	\frac{2}{\log 2} 
	 \left(1 + t^{\frac{1}{2}} c_{\Holderindexl} +t  \right)
	 \text{ and }  c_{\Holderindexl}=2\left(\| \sigma \|_{\infty}  + \sup_{n \in \bN}\| \sigma_n \|_{\infty} \right)  \| f \|_{\infty}.
\]
\end{remark}

\begin{proof}
As this proof concerns only $X^\infc$ and $X_n^{\infc}$, we drop the superscripts $\infc$ hereafter. 

By the argument used to prove Lemma \ref{lemma for classical stability problems}, where $2\Holderindexl+1= 2\Holderindex$, 
it is sufficient to prove that for some ${C_{} (f)}_t$, where 
\begin{align*}
 &\bE \left( \int_0^{t} \frac{d \langle \partSig \rangle_s }{ |\partY (s) |^{2\Holderindexl+1}} 1_{ (\Yp_n, \Yq_n)} (|\partY (s)|)  \right)
	\\&\leq\bE \left( \int_0^t  {\left| f ( X (s) ) - f ( X_n (s) ) \right| } { |\partY (s)|^{-1}}   1_{ (\Yp_n, \Yq_n)} (|\partY (s)|)  ds \right) =:{C_{} (f)}_t, 
\end{align*}
there exists some $c_{\Holderindexl} >0$ independent of $n \in \bN$ and $t\geq 0$ such that  
\[
{C_{} (f)}_t \leq  \infc^{-2} c_{\Holderindexl} t^{\frac{1}{2}}. 
\]

We now choose a sequence $\{ f_l (x) \}_{l \in \bN}$ of uniformly bounded (in $l$ and $x$) increase functions in $C^1 (\bR)$ such that 
$f_l (x) \rightarrow f (x) \ (l \rightarrow \infty)$ for any continuous point $x$ of $f$. 
The set $D$ of discontinuity points in $f$ is at most countable.  
The inequality $\infc \leq \sigma^\infc (x) $ for $x \in \bR$ implies that $t \leq  \infc^{-2}    \langle X \rangle_t \ (t \geq 0) $. 
Similarly, we have $t  \leq  \infc^{-2} \langle  X_n \rangle_t \ (t \geq 0)$. 
Thus, by the occupation times formula  for $X$ and $X_n$ (cf.~\cite[page 224, (1.6) Corollary]{RevuzYor1999}), we obtain 
\begin{align*}
&{\rm Leb} \left\{ 0 \leq s \leq t   :   X(s)   \in D \right\}+{\rm Leb} \left\{ 0 \leq s \leq t     :   X_n (s)  \in D \right\}
\\& =\int_{0}^{t}  1_{D} (X(s))  ds + \int_{0}^{t} 1_{D} (X_n (s))  ds 
\\&\leq \infc^{-2} \int_{0}^{t} 1_{D} (X(s))  d \langle X \rangle_s + \infc^{-2}\int_{0}^{t}  1_{D} (X_n (s)) d \langle X_n \rangle_s 
\\&=\infc^{-2}  \int_{D} L_t^a (X ) da +\infc^{-2}\int_{D} L_t^a (X_n ) da =0, 
\end{align*}
where 
${\rm Leb}$ is the Lebesgue measure, 
and $L_t^a (X_n )$ denotes the local time of $X_n $ accumulated at $a$ until time $t$. 
Therefore, we have 
\[
\bP \left( 
\lim_{l \to \infty}\left(  f_l ( X (s) ) - f_l ( X_n (s) )  \right)  =f ( X (s) ) - f ( X_n (s) ) \right)=1 
\]
for almost all $s \leq t$. It follows that ${C_{} (f)}_t =\lim_{l \to \infty} {C_{} (f_l)}_t $, where 
\[
{C_{} (f_l)}_t := 	
\bE \left( \int_0^t  {\left| f_l ( X (s) ) - f_l ( X_n (s) ) \right| } { |\partY (s)|^{-1}}   1_{ (\Yp_n, \Yq_n)} (|\partY (s)|)  ds \right) .
\]

For $\Ztheta$ in $[0,1]$, set $\YorZ^{\Ztheta}_n (t):= \Ztheta X (t)  + (1-\Ztheta)  X_n ( t )$ and 
$\eta^\Ztheta_n (t)(\omega):= \Ztheta \sigma (X (t) (\omega))  + (1-\Ztheta) \SDEa_n ( t, X_n  (\omega) ) \ (t \geq 0, \ \omega \in \Omega) $. 
Then we have 
\[
\YorZ^{\Ztheta}_n (t)= \YorZ^{\Ztheta}_n (0) +\int_0^t \eta^\Ztheta_n (s) dB_s 
\]
and $\left| f_l ( X (t) ) - f_l ( X_n (t) )   \right| =|\partY (t)| \int_0^1  f_l^{'}(\YorZ^{\Ztheta}_n (t))  d\Ztheta \  (t \geq 0)$. 
It follows that 
\begin{align*}
{C_{} (f_l)}_t =
	\bE \left(
	\int_{0}^{t}  \left(  \int_0^1 f_l^{'}(\YorZ^{\Ztheta}_n (s)) d\Ztheta   \right)^{}  1_{ (\Yp_n, \Yq_n)} (|\partY (s)|)
		ds \right). 
\end{align*}
As $u\infc+(1-u)\infc=\infc \leq \eta^\Ztheta_n (t, \omega) \ (t \geq 0, \ \omega \in \Omega )$, we have 
$ s \leq \infc^{-2} \langle \YorZ^{\Ztheta}_n \rangle_s \ (s \geq 0) $. 
By Jensen's inequality and the occupation times formula for $\YorZ^{\Ztheta}_n$, we have 
\begin{align*}
{C_{} (f_l)}_t
	&\leq   \bE \left( \int_{-\infty}^{\infty}  
		\left( \int_0^1  f_l^{'}(a)      L_t^a (\YorZ^{\Ztheta}_n)   \infc^{-2}  d\Ztheta \right) da \right).
\end{align*}

Now, by the Meyer-Tanaka formula, 
for $a \in \bR$ and $t \geq 0$, we obtain $\bP$-almost surely
\begin{align*}
L_t^a (\YorZ^{\Ztheta}_n) 
&=  |\YorZ^{\Ztheta}_n (t) -a|- |\YorZ^{\Ztheta}_n (0)  -a| -\int_0^t {\rm sgn} (\YorZ^{\Ztheta}_n(s)-a) dX (s)
\\ &\leq | \YorZ^{\Ztheta}_n(t)  - \YorZ^{\Ztheta}_n(0)  | -\int_0^t {\rm sgn} (\YorZ^{\Ztheta}_n (s)-a) \sigma^\Ztheta_n (s)  dB_s. 
\end{align*}
As $\| \sigma \|_{\infty}=\sup_{x \in \bR}| \sigma (x) |\leq |\sigma (0)| +  \| f \|_{\infty}<\infty$, 
we have for every $t \geq 0$
\[
\bE ( | X (t) | ) \leq \left[ \bE ( \int_0^t \left\{ \sigma (X (s)) \right\}^2 ds )\right]^{\frac{1}{2}} \leq t^{\frac{1}{2}} \| \sigma \|_{\infty} < \infty. 
\]
In the same way, assuming $\sup_{n \in \bN}\| \sigma_n\|_\infty < \infty$, we have 
\[
\bE ( | X_n (t) | )  \leq  t^{\frac{1}{2}} \sup_{n \in \bN}\| \sigma_n \|_{\infty}  < \infty.  
\]
It follows that for every $s \geq 0$
\begin{align*}
\bE (| \YorZ^{\Ztheta}_n(s)-\YorZ^{\Ztheta}_n(0) |)
\leq  t^{\frac{1}{2}}  \left( \| \sigma \|_{\infty}  + \sup_{n \in \bN}\| \sigma_n \|_{\infty} \right) 
< \infty. 
\end{align*}
Then 
\[
\int_0^t {\rm sgn} (\YorZ^{\Ztheta}_n(s)-a)  \sigma^\Ztheta_n   (s) dB_s
\]
is a martingale vanishing at $0$. 
Thus we have 
\begin{align*}
{C_{} (f_l)}_t &\leq 
	2 \infc^{-2} \left(\| \sigma \|_{\infty}  + \sup_{n \in \bN}\| \sigma_n \|_{\infty} \right)
	  t^{\frac{1}{2}}   \| f_l \|_{\infty}.
\end{align*}
Letting $l \to \infty$ 
and denoting 
$c_{\Holderindexl}=2\left(\| \sigma \|_{\infty}  + \sup_{n \in \bN}\| \sigma_n \|_{\infty} \right)  \| f \|_{\infty}$
, we conclude that 
\begin{align*}
	{C_{} (f)}_t \leq \infc^{-2}  c_{\Holderindexl} t^{\frac{1}{2}}< \infty.
\end{align*}
\end{proof}

As shown in Theorem \ref{theorem for classical stability problems in sup norm}, 
we derive the inequality in term of the $L^2$-sup-norm as follows. 
\begin{lemma}\label{gNLupper2} 
Suppose that the assumption of Lemma \ref{gNLupper} holds.  
Then, for  any $T>0$, there exists a $T$-dependent $C_4 (T)>0 $ 
such that for every $n \in \bN$ satisfying $\Delta_n <2^{-(\Holderindexl+\frac{1}{2}) }$ and $0<\infc <1$, we have
\begin{align*}
 \bE (\sup_{0 \leq t \leq T}|X^{\infc} (t) -X_n^{\infc} (t) |)^2 \leq 
   \infc^{-3}  C_4 (T) \times 
	\begin{cases}
		\left( -\log  \Delta_n\right)^{-\frac{1}{2}} \ & \text{if } \Holderindexl=0,  \\
		\Delta_n^{(1-1/(2\Holderindexl+1)) } \ &\text{if } \Holderindexl \in (0,\frac{1}{2}]. 
	\end{cases}
\end{align*}
\end{lemma}
\begin{remark}\label{Rem:C4}
$C_4 (T)$ satisfies $C_4 (T) \leq 4 ( c_{\Holderindexl}T^{\frac{1}{2}}+ 4\| f \|_{\infty}^2T+2T) \sqrt{ C_3 (T)}$, 
and $C_3 (T)$ is given by Lemma \ref{gNLupper}. 
\end{remark}

\begin{proof}
Again, this proof concerns only $X^\infc$ and $X_n^{\infc}$, 
so we drop the superscripts $\infc$. 

For an arbitrary $r>0$ and every $s \geq 0$, we obtain $\bP$-almost surely 
\begin{align*}
	&\left\{ \sigma (X(s) ) - \sigma_n (X_n (s) )  \right\}^2
	=\left\{ \sigma (X(s) ) - \sigma (X_n (s) ) +\sigma (X_n (s) ) - \sigma_n (X_n (s) )  \right\}^2 
	\\&\leq 2\left\{ \sigma (X(s) ) - \sigma (X_n (s) )  \right\}^2 +2\Delta_n^2
	\\&= 2\left\{ \sigma (X(s) ) - \sigma (X_n (s) )  \right\}^2 
		\left( 1_{\left\{ |X(s)- X_n (s)| <r \right\} } + 1_{\left\{ |X(s)- X_n (s)| \geq r \right\} } \right)
		+2\Delta_n^2
	\\&\leq r^{2\Holderindexl}  \left| f ( X (t) ) - f ( X_n (t) )   \right|   1_{\left\{ |X(s)- X_n (s)| <r \right\} } 
	+ 4 \| f \|_{\infty}1_{\left\{ |X(s)- X_n (s)| \geq r \right\} }+2\Delta_n^2. 	
\end{align*}
Following the estimation of ${C_{} (f)}_t$ in the proof of Lemma \ref{gNLupper}, we obtain 
\begin{align*}
\bE ( \int_{0}^T \left\{ \sigma (X(s) ) - \sigma_n (X_n (s) )  \right\}^2 ds )
 	\leq  r^{2\Holderindexl+1} {C_{} (f)}_T + \int_0^T 4 \| f \|_{\infty} r^{-1} \bE (|X(t)-X_n (t)|) +2\Delta_n^2 ds. 
\end{align*}
Thus, selecting $r$ such that 
\[
r =
	\left(  \sup_{0 \leq t \leq T} \bE (|X(t)-X_n (t)|)  \right)^{\frac{1}{2}},
\]
we obtain 
\begin{align*}
\bE ( \int_{0}^T \left\{ \sigma (X(s) ) - \sigma_n (X_n (s) )  \right\}^2 ds )
 	\leq (   \infc^{-2} {c_{\Holderindexl}} T^{\frac{1}{2}} +4 \| f \|_{\infty}T+2 T)
	\left(  \sup_{0 \leq t \leq T} \bE (|X(t)-X_n (t)|)  \right)^{\frac{1}{2}}. 
\end{align*}
Again, by Doob's maximal inequality
\begin{align*}
 \bE (\sup_{0 \leq t \leq T}|X(t) -X_n (t) |)^2 
 	\leq 4 
	(   \infc^{-2} {c_{\Holderindexl} T^{\frac{1}{2}}} +4 \| f \|_{\infty}T+2 T)
	 \left(  \sup_{0 \leq t \leq T} \bE (|X(t)-X_n (t)|)  \right)^{\frac{1}{2}}
\end{align*}
where $ \infc^{-2} >1$ and 
\begin{align*}
 \left(  \sup_{0 \leq t \leq T} \bE (|X(t)-X_n (t)|)  \right)^{\frac{1}{2}}
 	\leq & \infc^{-1}  
	\sqrt{ C_3 (T) }\times 
	\begin{cases}
		\left( -\log  \Delta_n\right)^{-\frac{1}{2}} \ & \text{if } \Holderindexl=0,  \\
		\Delta_n^{1-(1/(2\Holderindexl+1))} \ &\text{if } \Holderindexl \in (0,\frac{1}{2}].
	\end{cases}
\end{align*}
Thus, we obtain the desired conclusion. 
\end{proof}
Finally, we prove Theorem \ref{gNLupper in sup norm}. 
\begin{proof}[Proof of Theorem \ref{gNLupper in sup norm}]
As mentioned in Example \ref{positive sigma and PU}, 
$\sigma$ is H\"older-continuous with exponent $\Holderindex \in [0,1)$; 
thus for $x,y \in \bR$ with $|x-y| \leq 1$, it follows that 
\[
\left| \sigma(x) -  \sigma(y) \right| \leq 	|x-y|^{\frac{\Holderindex}{2}} 	\left| \sigma(x) -  \sigma(y) \right|^{\frac{1}{2} }. 
\]
Moreover, $\sigma$ is a function with finite variation. 
Thus, applying Lemma \ref{gNLupper2}  when $\Holderindexl=\frac{\Holderindex}{2}$, 
for any $T>0$, there exists a $T$-dependent $C_4 (T)>0 $ such that 
for every $n \in \bN$ satisfying $\Delta_n <2^{-\frac{1}{2}(\Holderindex+1) }$, 
\begin{align*}
 \bE (\sup_{0 \leq t \leq T}|X^{\infc} (t) -X_n^{\infc} (t) |)^2 \leq  
 	 \infc^{-3}   C_4 (T) \times 
	\begin{cases}
		\left( -\log  \Delta_n\right)^{-\frac{1}{2}} \ & \text{if } \Holderindex=0,  \\
		\Delta_n^{(1-1/(\Holderindex+1)) } \ &\text{if } \Holderindex \in (0,1]. 
	\end{cases}
\end{align*}
This concludes the proof. 
\end{proof}

\section{The Cantor Diffusion}\label{Devil's staircases diffusion}
To demonstrate the strong convergence rate, 
we consider a slight generalization of an iteratively constructed Cantor function 
called the middle-$\lambda$ Cantor function, (cf.~ \cite[Chapter 4]{cantrolDovgoshey20061}).  
Let $\lambda$ be in $(0,1)$. For all $n \in \bN \ \cup  \{0 \} $, we define 
\begin{align*}
\Cantor_{n+1} (x) =
	\begin{cases}
	\frac{1}{2} \Cantor_{n} (\frac{2}{1-\lambda} x) \ &\text{if }0 \leq x <\frac{1-\lambda}{2} 
	\\ \frac{1}{2}  \			  &\text{if } \frac{1-\lambda}{2} \leq x < \frac{1+\lambda}{2} 
	\\ \frac{1}{2}  + \frac{1}{2} \Cantor_{n} ( \frac{2}{1-\lambda} x- \frac{1+\lambda}{1-\lambda} ) \ &\text{if } \frac{1+\lambda}{2} \leq x \leq1 , 
	\end{cases}
\end{align*} 
where $\Cantor_0 : [0,1] \to \bR$ is an arbitrary function. 
Let $\mathscr{M} [0,1]$ be the Banach space of all uniformly bounded real-valued functions on $[0,1]$ with a supremum norm. 

\begin{lemma}[Middle-$\lambda$ Cantor function]\label{middle-lambda Cantor function}
There exists a unique element $\Cantor$ of $\mathscr{M} [0,1]$ such that 
\begin{align}\label{Def: Cantor sigma}
\Cantor (x) =
	\begin{cases}
	\frac{1}{2} \Cantor (\frac{2}{1-\lambda} x) \ &\text{if } 0 \leq x <\frac{1-\lambda}{2} 
	\\ \frac{1}{2}  \			  &\text{if } \frac{1-\lambda}{2} \leq x < \frac{1+\lambda}{2} 
	\\ \frac{1}{2}  + \frac{1}{2} \Cantor ( \frac{2}{1-\lambda} x- \frac{1+\lambda}{1-\lambda} ) \ &\text{if }\frac{1+\lambda}{2} \leq x \leq1 . 
	\end{cases}
\end{align} 
If $\Cantor_0 \in \mathscr{M} [0,1]$ then 
the sequence $\{ \Cantor_n \}_{n=0}^{\infty}$ converges uniformly to 
the middle-$\lambda$ Cantor function $\Cantor$. 
\end{lemma}
\begin{proof}

Define a map $H : \mathscr{M} [0,1] \to \mathscr{M} [0,1]$ as 
\begin{align*}
H (g) (x) =
	\begin{cases}
	\frac{1}{2} g (\frac{2}{1-\lambda} x) \ &\text{if } 0 \leq x <\frac{1-\lambda}{2} 
	\\ \frac{1}{2}  \			  &\text{if } \frac{1-\lambda}{2} \leq x < \frac{1+\lambda}{2}
	\\ \frac{1}{2}  + \frac{1}{2}g( \frac{2}{1-\lambda} x- \frac{1+\lambda}{1-\lambda} ) 
	\ &\text{if } \frac{1+\lambda}{2} \leq x \leq1. 
	\end{cases}
\end{align*} 
As
\[
\| H (g_1) - H (g_2 ) \| \leq \frac{1}{2} \| g_1 - g_2 \|,
\]
where $g_1, g_2 \in  \mathscr{M} [0,1]$ and $\| \cdot \|$ denotes the norm in $ \mathscr{M} [0,1]$ 
and $H$ is a contraction map on the complete space $ \mathscr{M} [0,1]$. 
Consequently, by the Banach theorem, $H$ has a unique fixed point $g_\infty$ such that $g_\infty = H (g_\infty )$ 
and $\Cantor_n$ converges to $g_\infty$ with the supremum norm. 
It follows from the definition of $H$, $(\ref{Def: Cantor sigma})$ holds.  
Hence, $H (\Cantor ) =\Cantor$ and the uniqueness imply that $g_\infty =\Cantor$. 
\end{proof}
Hereafter, 
we assume that 
\[
\Cantor_0 \in \mathscr{M} [0,1]. 
\]

An extended Cantor's function $\overline{\Cantor}$  is defined as follows: 
\[
\overline{\Cantor} (x) = 
\begin{cases} 
0 \ &\text{if }x < 0 \\
\Cantor(x) \ &\text{if } 0 \leq x \leq 1 \\
1 \ &\text{if }x > 1.
\end{cases} 
\]
In the same way, we define $\overline{\Cantor}_n$. 
\begin{lemma}\label{Cantor vn}
For every $x\in \bR$ and $n \in \bN$, we have 
\begin{align*}
 | \overline{\Cantor} (x) - \overline{\Cantor}_n (x) |  \leq 2^{-n+1}\sup_{x \in \bR}|  \overline{\Cantor}_0 (x) -  \overline{\Cantor}_1 (x) | . 
\end{align*}
\end{lemma}
\begin{proof}
For every $n, m \in \bN$ with $m>n$ and $x \in [0,1]$, we have 
\begin{align*}
\begin{split}
|\Cantor_{m} (x)  -\Cantor_{n} (x) | 
	&\leq \sum_{l=n}^{m-1} \ |\Cantor_{l+1} (x)  -\Cantor_{l} (x) | 
	\\&\leq \sum_{l=n}^{m-1}2^{-l} \cdot  \sup_{x \in [0,1]} |\Cantor_{1} (x)  -\Cantor_{0} (x) | 
	\\&\leq\left( 2^{-n+1}- 2^{-m+1} \right) \cdot \sup_{x \in [0,1]} |\Cantor_{1} (x)  -\Cantor_{0} (x) | . 
\end{split}
\end{align*}
If $x <0$ or $x >1$, then we have $|\Cantor_{m} (x)  -\Cantor_{n} (x) | = |\Cantor_{1} (x)  -\Cantor_{1} (x) |=0$. 
Letting $m \to \infty$ we obtain the result. 
\end{proof}
The middle-$\lambda$ Cantor function $\Cantor$ satisfies the following H\"older condition 
of order ${ \log{2} }/ { \left( \log{2} - \log (1-\lambda)  \right) }$: 
\begin{lemma}\label{C is Holder continuous}
Let $\lambda \in (0,1)$. 
The middle-$\lambda$ Cantor function $\Cantor$ is a H\"older-continuous function such that 
for every $x, y \in [0,1]$, 
\begin{align*}
 | \Cantor (x) - \Cantor (y) |  \leq  |x-y|^{\Holderindex_\lambda}, 
\end{align*}
where $\Holderindex_\lambda :=  { \log{2} }/ { \left( \log{2} - \log (1-\lambda)  \right) } $. 
Moreover, for every $x,y \in \bR$
\begin{align*}
| \overline{\Cantor} (x)  -\overline{\Cantor}_n (y) | 
	&\leq  |x-y|^{\Holderindex_\lambda} + 2^{-n+2}.  
\end{align*}
\end{lemma}
\begin{proof}
If $\overline{\Cantor_0}$ is selected to satisfy 
\[
\overline{\Cantor_0} (x) =	\begin{cases}
	0 \ &\text{if } x <0 
	\\ 1 \ &\text{if }1 \leq x . 
	\end{cases}
\]
It follows that 
\begin{align*}
\overline{\Cantor}_{n+1} (x) =
	\begin{cases}
	\frac{1}{2} \overline{\Cantor}_{n} (\frac{2}{1-\lambda} x) \ &\text{if } x <\frac{1+\lambda}{2} 
	\\ \frac{1}{2}  + \frac{1}{2} \overline{\Cantor}_{n} ( \frac{2}{1-\lambda} x- \frac{1+\lambda}{1-\lambda} ) \ &\text{if }\frac{1-\lambda}{2} \leq x. 
	\end{cases}
\end{align*}
Following the argument of \cite{Dobos1996zbMATH00955891}, 
it is easily shown that $\overline{\Cantor}_n$ is a subadditive function for every $n \in \bN$. 
Therefore, $\overline{\Cantor}$ is also subadditive. 
As $\overline{\Cantor}$ is also increasing, continuous, subadditive and 
satisfies $\overline{\Cantor } (0) =0$,  
$\overline{\Cantor}$ satisfies the first modulus of self-continuity:  for every $x,y \in [0,1]$, we have
\[
 | \overline{\Cantor} (x) - \overline{\Cantor} (y) |  \leq \overline{\Cantor}(| x-y |) 
\]
By \cite[Lemma 2]{cantorGorin_zbMATH02123795}, we have
$
\Cantor (h)   \leq h^{\Holderindex_\lambda} \ ( h \in [0,1] )$; 
thus, $\Cantor$ is Holder continuous. 
Furthermore, for every $x,y \in \bR$,  we have 
\[
 | \overline{\Cantor} (x) - \overline{\Cantor} (y) |  \leq  |x-y|^{\Holderindex_\lambda}.  
\]
Thus, by Lemma \ref{Cantor vn}, we obtain 
\begin{align*}
| \overline{\Cantor} (x)  -\overline{\Cantor}_n (y) | 
	&=| (\overline{\Cantor} (x)  -\overline{\Cantor} (y))  + (\overline{\Cantor} (y)  - \overline{\Cantor}_n (y))     |\\	
	&\leq  |x-y|^{\Holderindex_\lambda} + 2^{-n+2}.  
\end{align*}
\end{proof}

Let $X$ be a solution of the SDE $(\ref{sde0})$ with the diffusion coefficient $\overline{\Cantor}$, 
\begin{align}\label{Cantor1}
X (t) = X (0)+  \int_{0}^t \overline{\Cantor} (X (s) ) dB_s, 
\end{align}
and let $X_n$ be a solution to the SDE $(\ref{sde1n})$ with the diffusion coefficient $\overline{\Cantor}_n$, 
\begin{align}\label{Cantor1n}
X_n (t) = X (0)+  \int_{0}^t \overline{\Cantor}_n ( X_n (s) ) dB_s. 
\end{align}


As $\overline{\Cantor} $ is continuous, 
the SDEs $(\ref{Cantor1})$ have weak solutions, see Skorohod \cite{SkoroMR0185620}. 
By Proposition \ref{proposition iff} and Lemma \ref{C is Holder continuous}, 
pathwise uniqueness holds if and only if $\lambda \in (0, \frac{1}{2}]$. 
Then, it follows from Yamada-Watanabe theorem \cite{YW} that $X$ is a unique strong solution. 
\subsection{Degenerate diffusion coefficients $(0 < \lambda \leq \frac{1}{2})$}
Firstly, we show Corollary \ref{Cantor estimation for Holder} as follows. 
\begin{proof}[The proof of Corollary \ref{Cantor estimation for Holder}]
By the construction of $\Cantor_n$, we obtain that for all $n \in \mathbb{N}$, 
\[
\Cantor_{n} (x) =
	\left(\frac{1}{2}\right)^{n-1} \Cantor_{0} \left( \left(\frac{2}{1-\lambda}\right)^{n-1} x \right), \ 0 \leq x <\frac{1-\lambda}{2}.
\]
Since $\Cantor_0$ is a H\"older continuous satisfying 
\begin{align*}
\lim_{x \downarrow 0} \int_{x}^{1} \frac{1-y}{(\Cantor_{0} (y))^2} dy =\infty. 
\end{align*}
Thus, pathwise uniqueness holds, which implies that $X_n$ is also unique strong solution. 

Now, applying Theorem \ref{theorem for classical stability problems in sup norm} 
to 
\[
\Delta_n =  \sup_{x \in \bR}| \overline{\Cantor} (x) - \overline{\Cantor}_n (x) |  \leq 2^{-n+2},  
\]
we obtain the result. 
\end{proof}

\begin{remark}
The Cantor function $\Cantor_n$ with $\Cantor_0 (x) =0$ for $x \in [0,1]$ 
is not continuous. 
Moreover, there does not exist a weak solution 
of (\ref{Cantor1n}) since for all $x \in I$, 
$\sigma^{-2}$ is not integrable over a neighborhood of  $x$, see \cite[Theorem 2.2]{Engelbert:1985aa}.
\end{remark}

\subsection{Non degenerate and Degenerate diffusion coefficients $(0 < \lambda < 1)$}
Next, let us consider the middle-size-$\lambda$ Cantor function with $\lambda \in (0,1)$. 

\begin{proof}[The proof of Corollary \ref{Cantor estimation for a generalized Nakao-Le Gall}.]
For $\infc \geq 0$, we denote for every $x \in \bR$
\begin{align*}
\overline{\Cantor}^{\infc} (x) := \infc +\overline{\Cantor} (x)  \text{ and }  \overline{\Cantor}_n^{\infc} (x) := \infc +\overline{\Cantor}_n (x). 
\end{align*}
The condition $\infc >0$ implies the existence and uniqueness of weak solutions of the following SDEs, 
\begin{align*}
   X^{\infc} (t) &= X (0)+  \int_{0}^t \overline{\Cantor}^{\infc}  (X^{\infc} (s) ) dB_s, 
\\ X^{\infc}_n (t) &= X (0)+  \int_{0}^t \overline{\Cantor}^{\infc}_n  (X^{\infc}_n (s) ) dB_s, \ n \in \mathbb{N}. 
\end{align*}
Moreover, 
as  $\Cantor$ and $\Cantor_n$ are monotone increasing functions for $n \in \mathbb{N}$, 
pathwise uniqueness holds. 
Finally, 
applying Theorem \ref{gNLupper in sup norm}, we obtain the desired result. 
\end{proof}

\begin{remark}
The Cantor function $\Cantor_n$ with $\Cantor_0 (x) =0$ for $x \in [0,1]$ 
is not continuous, 
however, $\Cantor_n^{\infc}$ is belong to a generalized Nakao-Le Gall condition. 
Therefore, its convergence rate can be estimated as Corollary \ref{Cantor estimation for a generalized Nakao-Le Gall}. 
\end{remark}

\begin{remark}
For every $\Holderindex_{\lambda} \in (0,1)$, 
\[
2^{-n(1-1/{(\Holderindex_{\lambda}+1)} )}  \leq 2^{-n(1-1/{2\Holderindex_{\lambda}} )}. 
\]
Moreover, the constant value $C(T)$ is depend only on $T$ and independent $\lambda$, 
see Remark \ref{Rem:C4}. 
Comparing the result of Theorem \ref{Cantor estimation for Holder}, 
the sharpness and stability of the above estimation is confirmed. 
\end{remark}


\section{The Fokker-Planck-Kolmogorov equations}\label{The Fokker-Planck-Kolmogorov equation}
Let $T>0$, let $f$ be a polynomial growth function such that 
there exists $r>0$ and $L>0$, $| f(x) | \leq L (1+|x|)^{2r}, \ x \in [0, \infty)$, 
and let $a$ be a real-valued locally bounded Borel function. 
In this section, we consider a fundamental solution of the so-called Fokker-Plank equation, 
\begin{align}\label{PDE1}
	\begin{cases}
	\displaystyle \frac{\partial}{\partial t} v(t,x)
		=\frac{1}{2}\frac{\partial^2}{\partial x^2} \left(a (x) v(t,x) \right), \ &(t,x) \in (0, T] \times (0,\infty), 
	\\ \displaystyle \lim_{t \to 0 } v(t,x)=f(x), \ x \in (0,\infty). 
	\end{cases}
\end{align}

For the unbounded domain $(0, T] \times (0,\infty)$, 
it is known that 
there exists a fundamental solution of \eqref{PDE1}
if $a \in C^2_b (0, \infty)$, and $a'$ and $a''$ are bounded H\"older continuous and uniformly parabolic: 
There exist $C, c>0$ such that 
\begin{equation}\label{parabolic condition}
0<c \leq a(x)  \leq C  < \infty, \ x \in [0, \infty).
\end{equation}
For details, see Friedman's book \cite[Theorem 15 in Chapter 1]{friedman2013partial}. 
Now, let $\mu_t $ be the law of solution of $X (t)$ in $(0,\infty)$ to the equation 
\begin{align*}
X (t) -X (0) = \int_0^t \sqrt{a  (X (s))} dB_s,  \ X(0)=x  \in (0, \infty). 
\end{align*}
By It\^o's formula, $\mu_t $ solves the above Fokker-Plank equation in the sense of distribution. 
In particular, if the law of $X(t)$ is smooth absolutely continuous with respect to Lebesgue measure, 
i.e., $\mu_t (dy) =p(t, x ,y) dy$, then $v (t,x)=\int_{0}^{\infty} p(t,x,y) f(y) dy$ solves the PDE \eqref{PDE1}. 

Conversely, 
if there exists a solution $X(t)$ of the SDE whose smooth coefficients do not satisfy \eqref{parabolic condition} 
and the law of $X(t)$ is absolutely continuous with respect to Lebesgue measure, 
then it provides a solution of the PDE \eqref{PDE1}. 
In fact, by Proposition \ref{proposition iff}, 
under weak conditions on the degenerate coefficients, 
we show that there exists a smooth fundamental solution $v (t,x)$ 
using the result of Malliavin calculus. 
\begin{theorem}
Let $c \in (0, \infty)$ and let $\sigma$ be a real-valued Borel function 
that is monotonically increasing in $(0,\infty)$ 
and $\sigma^{-2}$ is integrable over a neighborhood of $x \in (0,\infty)$. 
If $\sigma \in C^\infty (0, \infty)$ and 
\begin{align*}
\lim_{x \downarrow 0} \int_{x}^{c} \frac{c-y}{(\sigma (y))^2} dy =\infty, 
\end{align*}
then, 
there exists a function $p = p(t,x,y):(0, T] \times (0,\infty)^2 \to \mathbb{R}$ such that 
\[
v (t,x)=\int_{(0, \infty)} p(t,x,y) f(y) dy, \ \ (t,x) \in [0, T] \times (0, \infty)
\]
is a smooth fundamental solution to the following Fokker-Planck-Kolmogorov equation 
\begin{align*}
	\begin{cases}
	\displaystyle \frac{\partial}{\partial t} v(t,x)
		=\frac{1}{2}\frac{\partial^2}{\partial x^2} \left(\sigma (x)^2 v(t,x) \right), \ &(t,x) \in (0, T] \times (0,\infty), 
	\\ \displaystyle \lim_{t \to 0 } v(t,x)=f(x), \ x \in (0,\infty). 
	\end{cases}
\end{align*}
where $f$ is a polynomial growth function. 
\end{theorem}
\begin{proof}
It follows from \cite[Theorem 1]{Engelbert:1985aa} that 
there exists a weak solution of the following driftless SDE, 
\begin{align*}
X (t) -X (0) = \int_0^t \sigma  (X (s)) dB_s, \ X(0)=x \geq 0. 
\end{align*}

Now, by using Proposition \ref{proposition iff} 
and the Yamada-Watanabe theorem \cite{YW}, 
the solution is unique strong solution. 
As a consequence of the result of Malliavin calculus in \cite[Theorem 2.5 (a)]{MarcoStefanoAOP2011demarco2011}, 
we have that, for an arbitrary $T>0$, 
there exists a smooth density $p$ of  the law of $X (t)$, 
such that 
\begin{align*}
	\frac{\partial}{\partial t} p(t,x,y)
		=\frac{1}{2}\frac{\partial^2}{\partial y^2} \left(\sigma (y)^2 p(t,x,y) \right), \ &(t,x,y) \in (0, T] \times (0, \infty)^2. 
\end{align*}
Denote
\[
v (t,x)=\int_{0}^{\infty} p(t,x,y) f(y) dy, \ \ (t,x) \in (0, T] \times (0, \infty). 
\]
As $p(t,x,y) \to \delta_{x} (y)$ as $t \to 0$, it follows that 
\[
\lim_{t \to 0 }v (t,x)=f(x),  \ x \in (0, \infty). 
\]
Thus, we obtain the desired result. 
\end{proof}

\begin{remark}
For an example, 
the Fokker-Planck equation \eqref{PDE1} such that 
\[
a(x) = |x|^{2 \alpha}, \ \alpha \in [{1}/{2},1] \text{ and } x \in (0, \infty), 
\]
 has a smooth fundamental solution, 
 as it does not satisfy hypotheses made in the PDE literature to get a weak or smooth solution, 
 e.g.~the uniformly parabolic condition \eqref{parabolic condition}. 
\end{remark}

\section*{Acknowledgements}
The authors would like to thank Toshio Yamada 
who inspired this study. 
We also thank Professor Dr.~Toshiro Watanabe for his useful comments and his constant encouragement, 
and Tomoyuki Ichiba at the University of California Santa Barbara for useful comments on the thesis. 
We are grateful to 
the anonymous associate editor and the anonymous reviewers 
who made significant suggestions to improve the quality of the paper. 

\providecommand{\bysame}{\leavevmode\hbox to3em{\hrulefill}\thinspace}
\providecommand{\MR}{\relax\ifhmode\unskip\space\fi MR }
\providecommand{\MRhref}[2]{%
  \href{http://www.ams.org/mathscinet-getitem?mr=#1}{#2}
}
\providecommand{\href}[2]{#2}

\end{document}